\documentclass[11pt,regno]{amsart}
\usepackage{euscript,amscd,amsgen,amsfonts,amssymb,latexsym}

\newcommand{\eqdef}{\stackrel{\scriptscriptstyle\rm def}{=}}

\newtheorem{theorem}{Theorem}
\newtheorem{remark}{Remark}

\newtheorem{proposition}{Proposition}

\newtheorem{lemma}{Lemma}

\def\vol{{\rm vol} }
\newcommand{\beha}{\begin{enumerate}}
\newcommand{\behe}{\end{enumerate}}
\renewcommand{\epsilon}{\varepsilon}

\newcommand{\cM}{\EuScript{M}}

\newcommand{\bR}{{\mathbb R}}

\newcommand{\bZ}{{\mathbb Z}}
\newcommand{\bN}{{\mathbb N}}

\newcommand{\cR}{{\mathcal R}}

\renewcommand{\det}{{\rm Jac\,}}

\def\1{1\!\!1}

\def\and{\text{ and }}

                        \def\^{\tilde}

\def\SPer{{\rm SPer}}
\def\SFix{{\rm SFix}}
\def\Fix{{\rm Fix}}

\def\1{1\!\!1}

\DeclareMathSymbol{\varnothing}{\mathord}{AMSb}{"3F}
\renewcommand{\emptyset}{\varnothing}
\begin{document}

\title[Equality of pressures]{Equality of pressures for diffeomorphisms preserving hyperbolic measures}
\begin{abstract}
For a diffeomorphism which preserves a hyperbolic measure the potential $\varphi^u=-\log|{\rm Jac\,} df|_{E^u}|$ is studied.  Various types of pressure of $\varphi^u$ are introduced. It is shown that these pressures  satisfy a corresponding variational principle.
\end{abstract}
\author{Katrin  Gelfert}\address{Current address: 
Department of Mathematics, Northwestern University, 2033 Sheridan Road, Evanston, IL 60208-2730, USA
}\email{gelfert@pks.mpg.de} 
\urladdr{http://www.pks.mpg.de/$\sim$gelfert}

\thanks{This research was supported by  the grant EU FP6 ToK SPADE2. The author is grateful to IM PAN Warsaw for the hospitality  and to C.~Wolf for discussions about suitable concepts of pressure. }
\maketitle

\section{Introduction}\label{secnoncon}

For a uniformly hyperbolic diffeomorphism $f$, the induced volume deformation $\varphi^u$ in the unstable subbundle over a compact $f$-invariant  set significantly characterizes the geometry of the set as well as the dynamics in its neighborhood. 
Under the hypothesis of uniform hyperbolicity, and particularly for hyperbolic surface diffeomorphisms, a large number of dynamical quantifiers such as, for example, fractal dimensions, Lyapunov exponents, and escape rates, are captured through the topological pressure of $\varphi^u$.
If  $f\colon M\to M$ is a $C^{1+\varepsilon}$ diffeomorphism on a Riemannian manifold $M$ (we assume that some Riemannian metric on $M$ is fixed) and  $\Lambda$
is an $f$-invariant locally maximal set such that $f|\Lambda$ is uniformly hyperbolic and satisfies specification (and hence is mixing), then the  topological pressure $P_{f|\Lambda}$ of  the function $\varphi^u(x)=-\log |\det df_x|_{E^u_x}|$  can be calculated through
\begin{multline}\label{vpvpvp}
P_{f|\Lambda}(\varphi^u) 
= \lim_{n\to\infty}\frac{1}{n}\log\sum_{f^n(x)=x}|\det df^n_x|_{E^u_x}|^{-1}\\
= \sup_{\mu\in\cM_{\rm E}}\left(h_\mu(f)-\int_\Lambda\log| \det df|_{E^u}| d\mu\right)
\end{multline}
(see~\cite{Bow:75} or~\cite{KatHas:95}).
Here the second equality (with the supremum taken over all ergodic $f$-invariant measures supported in $\Lambda$) simply follows from the variational principle by continuity of the function $\log|\det df|_{E^u}|\colon\Lambda\to\bR$. Note that the supremum in~(\ref{vpvpvp}) is in fact a maximum (see~\cite{Bow:75}). The measure which realizes this maximum is of unique importance from several different points of view. A particular case is given when the maximizing measure is the SRB (after Sinai Ruelle Bowen) measure of $f$, each of the terms in~(\ref{vpvpvp}) is zero, and $\Lambda$ is an attractor  (see~\cite{You:02} for further details and references). 

In the case of more general dynamical systems, the above quantities are likewise  important, in particular for the issue of existence of SRB measures. Note that the classical thermodynamic formalism, however, requires the potential to be  continuous. We point out that for a non-uniformly hyperbolic system (a system with non-zero integrated Lyapunov exponents) it is natural to consider potentials which are discontinuous: no continuous $df$-invariant subbundle $E^u\subset T_\Lambda M$ may exist and $x\mapsto-\log| \det df_x|_{E^u_x}|$ is in general only a measurable function.  In this paper we study appropriate modifications of any of the terms in~(\ref{vpvpvp}) by exploiting techniques which were developed by Katok~\cite{Kat:80} and Mendoza~\cite{Men:88} for dynamical systems with some non-uniformly hyperbolic behavior. We combine them with an approach of a non-additive version of the thermodynamic formalism, developed by~\cite{Pes:97,Fal:88,Bar:96} in particular for non-conformal systems. 

It is meaningful to consider a function $\varphi^u\eqdef-\log| \det df|_{E^u}|$ which is defined only on a certain subset of $\Lambda$. Pesin~\cite{Pes:97} developed an extension of the classical topological pressure to a pressure on sets which are not necessarily compact nor invariant, but his approach requires the potential functions to be continuous. Mummert~\cite{Mum:} discusses for example a pressure for non-continuous potentials and provides a meaningful generalization of $P_{f|\Lambda}(\varphi^u)$ in the case that $f$ preserves a hyperbolic measure. 
In~\cite{GelWol:} the so-called \emph{saddle point pressure} $P_{f,\rm SP}(\varphi^u)$ is introduced, which is entirely determined by the values of $\varphi^u$ on the periodic points of saddle type (see Section~\ref{sec:1} for the definition), and which generalizes the second term in~(\ref{vpvpvp}) in the case that such periodic points do exist. 
Notice that, by the multiplicative ergodic theorem, $h_\mu(f)+\int_\Lambda\varphi^ud\mu$ is well-defined for any ergodic $f$-invariant probability measures with a positive Lyapunov exponent.

The main result of this paper is to show that the equalities~(\ref{vpvpvp}) extend to more general maps, including $C^{1+\varepsilon}$ diffeomorphisms possessing hyperbolic invariant probability measures. Here we call an ergodic measure \emph{hyperbolic} or say it is  of \emph{saddle type} if it possesses at least one negative and one positive, and no zero Lyapunov exponents.
We say that $f$ is \emph{non-uniformly hyperbolic}  if every ergodic $f$-invariant measure is hyperbolic.

Paradigms of genuinely non-uniformly hyperbolic diffeomorphisms are given, for example, within the family of H\'enon maps (see~\cite{CaoLuzRio:}). 
Another perhaps simplest example is  provided by the figure-8 attractor (composed of homoclinic loops joining a fixed point of saddle type, see for example~\cite[p.~140]{Kat:80}). Even though here the maximal invariant set which is formed by the loops does not support any other invariant probability measures besides the Dirac measure supported at the saddle fixed point, it provides a basic plug for more sophisticated models. A similar attractor is used for example in~\cite{BalBonSch:99} where  it is inserted  by a smooth surgery into a uniformly hyperbolic set. The resulting compact set $\Lambda$ is invariant and locally maximal under a smooth diffeomorphism $f$ on a surface and $\Lambda$ is the support of a 
measure $\mu$ satisfying
\[
h_\mu(f)  + \int\varphi^ud\mu = 
\max_\nu\left( h_\nu(f)  + \int\varphi^ud\nu\right) <0,
\] 
where the maximum is taken over all ergodic $f$-invariant probability measures supported in $\Lambda$. 
Moreover, any ergodic measure supported in $\Lambda$ is hyperbolic, and hence $f|\Lambda$ is non-uniformly hyperbolic  in the sense introduced above, but not uniformly hyperbolic. 

In the following we always assume that $f\colon M\to M$ is a $C^{1+\varepsilon}$ diffeomorphism preserving a hyperbolic Borel probability measure supported on a  compact locally maximal $f$-invariant set $\Lambda\subset M$. Here by locally maximal we mean that there exists an open neighborhood $U\subset M$ of $\Lambda$ such that $\Lambda=\bigcap_{n\in\bZ}f^n(U)$. 
We denote by $\cM$ the set of all Borel $f$-invariant probability measures on
$\Lambda$, endowed with weak$*$ topology, and by $\cM_{\rm E}\subset \cM$ the subset of ergodic measures.

\begin{theorem}\label{Main}
	Let $f\colon M\to M$ be a $C^{1+\varepsilon}$ diffeomorphism and let $\Lambda\subset M$ be a compact locally maximal $f$-invariant set such that there exists a hyperbolic $f$-invariant Borel probability measure supported on $\Lambda$. Then
\begin{equation}\label{main}
\sup_{K\subset\Lambda}P_{f|K}(\varphi^u)
= P_{f,\rm SP}(\varphi^u)
= \sup_{\nu\in\cM_{\rm E}(f|\Lambda)\,\,{\rm hyperbolic}}
	\left(h_\nu(f)+\int_\Lambda\varphi^ud\nu\right)	,
\end{equation}
with the first supremum taken over all compact $f$-invariant hyperbolic sets $K\subset \Lambda$. 
\end{theorem}

Note that by the Ruelle inequality~(\ref{main}) is always non-positive. By Young~\cite[Theorem 4(1)]{You:90},~(\ref{main}) is bounded from above by the escape rate of volume from a small open neighborhood of $\Lambda$, which in turn is negative only if there is a certain amount of repulsion in $\Lambda$ and zero if $\Lambda$ is attracting. However, as indicated above,~(\ref{main}) can be negative even though that $\Lambda$ is an attractor. If~(\ref{main}) is zero and there exists a maximizing measure $\mu$ with $h_\mu(f)>0$, then  $\mu$ is a SRB measure, and conversely if there exists a SRB measure $\mu$, then by the Pesin formula for the entropy
\[
h_\mu(f) = -\int_\Lambda\varphi^ud\nu
= \lim_{n\to\infty}\int_\Lambda\log||( df^n)^\wedge||^{1/n}\,d\mu
\]
(again we refer to~\cite{You:02} for details and references).

Here, in the previous formula the last integral, $x\mapsto(df^n_x)^\wedge$ is a map between the full exterior algebras of the tangent spaces $T_xM$ and $T_{f^n(x)}M$, induced by $df^n_x$, and $||\cdot||$ is the operator norm induced by the Riemannian metric. In geometric terms $||(df^n_x)^\wedge||$ measures the maximum of volumes of images under $df^n_x$ of an arbitrarily $k$-dimensional cube, $1\le k\le\dim M$, of volume $1$. Notice that for fixed $n$
\begin{equation}\label{chi}
\varphi_n(x)\eqdef -\log|| (df^n_x)^\wedge ||
\end{equation}
is a H\"older continuous functions. Turning to the pressure of that \emph{continuous} potential, we can formulate the following result.  

\begin{theorem}\label{theorem:heu}
	Let $f\colon M\to M$ be a $C^{1+\varepsilon}$ diffeomorphism and let $\Lambda\subset M$ be a compact locally maximal $f$-invariant set such that every $f$-invariant Borel probability measure supported on $\Lambda$ possesses a positive Lyapunov exponent. Then
	\begin{equation}\label{eq:utw}
	\sup_{\mu\in\cM(f|\Lambda)}\left( h_\mu(f) + \int_\Lambda\varphi^ud\mu\right) 
	=  \sup_{n\ge 1}P_{f|\Lambda}\left(-\log||( df^n)^\wedge||^{1/n}\right).
\end{equation}
\end{theorem}

Note that the right hand side of~(\ref{eq:utw}) is independent of the Riemannian metric.
Further it is immaterial whether in the left hand side in~(\ref{eq:utw}) we take the supremum over measures in $\cM$ or in $\cM_{\rm E}$ since the entropy $h_\mu(f)$ and the integral both are affine functions of $\mu$.
We remark that for any example of a diffeomorphism with a dominated splitting, and in particular for any partially hyperbolic diffeomorphism, which has a uniformly expanding subbundle satisfies the assumptions in Theorem~\ref{theorem:heu}.

We introduce now a super-additive version $P_{f,\rm SP}(\Phi)$ of pressure of the \emph{sequence} $\Phi\eqdef (\varphi_n)_n$ of the H\"older continuous functions given by~(\ref{chi}) (see Section~\ref{sec:nonadd} for the definition of the pressure).

\begin{theorem}\label{theorem:2}
	Let $f\colon M\to M$ be a non-uniformly hyperbolic $C^{1+\varepsilon}$ diffeomorphism and let $\Lambda\subset M$ be a compact locally maximal $f$-invariant set. Then
	\begin{equation}\label{eq:www}
	P_{f,\rm SP}(\Phi)  
	= \sup_{n\ge 1}P_{f,\rm SP}\left(-\log||( df^n)^\wedge ||^{1/n}\right)
	 = 	\sup_{n\ge 1}P_{f|\Lambda}\left(-\log||( df^n)^\wedge||^{1/n}\right).
\end{equation}
\end{theorem}

Note that, by non-uniform hyperbolicity, the terms in~(\ref{eq:www}) are also equal to any of the terms in~(\ref{main}) and~(\ref{eq:utw}).
At the end of Section~\ref{sec:nonadd} we will discuss hypotheses, which are slightly weaker than non-uniform hyperbolicity, under which we can prove~(\ref{eq:www}) in the case of a surface diffeomorphism.

\begin{remark}{\rm
In~\cite{Bar:96}  the non-additive topological pressure $P_{f|\Lambda}(\Psi)$ of a general sequence $\Psi$ of continuous functions with respect to $f|\Lambda$ is studied and a non-additive version of the variational principle of the classical thermodynamic formalism is established. Applying this approach to the above setting, we can conclude that  
\[
\sup_{\nu\in\cM_{\rm E}(f|\Lambda) \,\,{\rm hyperbolic}}
	\left(h_\nu(f)+\int_\Lambda\varphi^ud\nu\right) \le
P_{f|\Lambda}(\Phi),
\] 
with the sequence $\Phi$ defined in~(\ref{chi})
(see~\cite[Section~1.3]{Bar:96} or~\cite{Pes:97}). A variational principle for $P_{f|\Lambda}(\Psi)$ in terms of an equality has been established in~\cite{Bar:96}, however only for a rather restrictive class of sequences $\Psi=(\psi_n)_n$ of continuous functions, to which $\Phi$ in general does not belong if $f|\Lambda$ is not hyperbolic. 
}\end{remark}

We now sketch the contents of the paper. In Section~\ref{sec:1} we review some concepts from smooth ergodic theory. 
 We also recall the definition of the pressure $P_{f,\rm SP}(\varphi^u)$ given in~\cite{GelWol:} and we introduce the pressure $P_{f,\rm SP}(\Psi)$ of a super-additive function sequence $\Psi$.  In Section~\ref{sec:nonadd}  we derive several properties of $P_{f,\rm SP}(\Psi)$ and prove  Theorems~\ref{theorem:heu} and~\ref{theorem:2}. The proof of Theorem~\ref{Main} is given in Section~\ref{sec:2}.

\section{Various types of pressure}\label{sec:1}

We first review some concepts from smooth ergodic theory and fix some notation.

\subsection{Notions from smooth ergodic theory}

Given a point $x\in\Lambda$ which is Lyapunov regular with respect to $f$ (see for example~\cite{Man:} for the definition and details on Lyapunov regularity),  there exist a positive integer $s(x)\leq \dim M$, numbers $\lambda_1(x)<\cdots<\lambda_{s(x)}(x)$, and a $df$-invariant splitting
\[
T_xM = \bigoplus_{i=1}^{s(x)}E^i_x
\]
such that for all $i=1,\ldots,s(x)$ and $v\in E^i_x\setminus \{0\}$ we have
\[
\lim_{n\to\pm\infty}\frac{1}{n}\log|| df^n_x(v)||
 =\lambda_i(x).
\]
We will count the values of the Lyapunov exponents $\lambda_i(x)$ with their multiplicity, i.e. we consider the numbers $\lambda_1(x)\le\cdots\le\lambda_{{\rm dim} M}(x)$.

By the Oseledets multiplicative ergodic theorem, given $\mu\in \cM$ the set of Lyapunov regular points has full measure and
$\lambda_i(\cdot)$ is $\mu$-measurable. We denote by
\begin{equation}\label{deflyame}
\lambda_i(\mu)\eqdef\int\lambda_i(x) d\mu(x).
\end{equation}
the Lyapunov exponents of the measure $\mu$. In particular, if for $\mu\in\cM_{\rm E}$ there is $1\le \ell=\ell(\mu)<\dim M$ such that
\[
\lambda_\ell(\mu)<0<\lambda_{\ell+1}(\mu),
\]
we say that $\mu$ is \emph{hyperbolic} or of \emph{saddle type}. 
In the following we always assume that there exists an $f$-invariant ergodic   Borel probability measure which is hyperbolic.

Define
\[
\chi(\mu)\eqdef \min_i\{| \lambda_i(\mu)|\}.
\]
For a Lyapunov regular point $x\in\Lambda$ let us denote through $E^u_x$ ($E^s_x$) the span of the subspaces of $T_xM$ that correspond to a positive Lyapunov exponent (a negative Lyapunov exponent) and let
\[
\varphi^u(x) \eqdef -\log| \det df_x|_{E^u_x}|.
\] 
If the subspace $E^u_x$ is empty, it is convenient to set $\varphi^u(x)\eqdef 0$.

We call a compact  $f$-invariant set $K\subset M$ \emph{hyperbolic} if there
exists a continuous $df$-invariant splitting of the tangent bundle $T_KM = E^u\oplus E^s$ and constants $c>0$ and $\lambda\in(0,1)$ such that
for every $x\in K$ and every $k\in\bN$
\[
|| df^k_x(v)|| \le c\lambda^k || v||\,\,{\rm for}\,\,{\rm all }\,\,v\in E^s_x,
|| df^{-k}_x(w)|| \le c\lambda^k || w||\,\,{\rm for}\,\,{\rm all }\,\,w\in E^u_x .
\]
The diffeomorphism $f$ is said to be \emph{uniformly hyperbolic} or \emph{Axiom A} if the periodic orbits are dense in the non-wandering set $\Omega(f)$ and if $f|\Omega(f)$ is hyperbolic. We call the diffeomorphism $f$ \emph{non-uniformly hyperbolic} if every  $f$-invariant  ergodic Borel probability measure is hyperbolic.

\subsection{The saddle point pressure}

We denote by $\cR$ the set of  Lyapunov regular points in $\Lambda$. We note that the subbundle  \( E^u\subset T_{\cR}M\) cannot always be extended  continuously  to $T_\Lambda M$, and in general there is no continuous function $\varphi$ with $\varphi|\cR = \varphi^u$.

We denote by $\SFix(f^n)$ the fixed points of $f^n$ which are saddle points, and by
$\SPer(f)=\bigcup_{n\ge 1} \SFix(f^n)$ the set of all saddle points.
We introduce a filtration of subsets of $\SPer(f)$ which are ordered according to the ``strength of hyperbolicity" of the saddle points. 
For $0<\alpha$, $0<c\leq 1$, and $n\in\bN$ we set
\begin{multline*}
\SFix_{(f,\alpha,c)}(f^n)\eqdef\\ 
\big\{ x\in\SFix(f^n)\colon
 \,\,{\rm for}\,\,{\rm all }\,\, k\ge 1, v\in E^s_{f^\ell(x)}, w\in E^u_{f^\ell(x)} ,0\le \ell\le n-1 \\
\quad || df^{-k}_{f^\ell(x)}(v)||\ge c e^{k\alpha}|| v||,\,\,
|| df^k_{f^\ell(x)}(w)|| \ge c e^{k\alpha}|| w|| \big\}.
\end{multline*}
If  $\alpha\geq \alpha',\,c\geq c'$, then
\begin{equation}\label{ni}
\SFix_{(f,\alpha,c)}(f^n) \subset \SFix_{(f,\alpha',c')}(f^n).
\end{equation}
Further
\[
\SPer(f)=\bigcup_{\alpha>0}\bigcup_{c>0}\bigcup_{n=1}^\infty
\SFix_{(f,\alpha,c)}(f^n).
\]
Note that in the above definition we use the sub-script $\SFix_{(f,\cdot,\cdot)}$ in order to emphasize with respect to which map $f$ the strength of hyperbolicity is considered, because in the following we will take into consideration also iterates of the diffeomorphism $f$.
Notice that for every $m\ge 1$ we have 
\begin{equation}\label{kubu}
\SFix_{(f^n,n\alpha,c\cdot d(n))}(f^{mn}) \subset
\SFix_{(f,\alpha,c)}(f^{mn}) \subset
\SFix_{(f^n,n\alpha,c)}(f^{mn}),
\end{equation}
where 
\begin{equation}\label{kubus}
d(n)\eqdef  
\min_{1\le \ell\le n-1}\left(e^{\ell\alpha} 
\min_{y\in\Lambda}\left\{ ||(df_y^\ell)^{-1}||, ||df_y^\ell ||\right\}\right).
\end{equation}

Clearly, every periodic point is Lyapunov regular,
and thus, $\varphi^u(x)$ is well-defined for every $x\in \SPer(f)$.
Let $0<\alpha$ and $0<c\leq 1$.
Define
\[
  Q_{f,\rm SP}(\varphi^u,\alpha,c,n) \eqdef
  \sum_{x\in \SFix_{(f,\alpha,c)}(f^n)} \exp\left( S_n\varphi^u(x)\right)
\]
if $\SFix_{(f,\alpha,c)}(f^n)\ne\emptyset$, where we use the notations $S_n\psi(x)\eqdef\psi(x)+\cdots +\psi(f^{n-1}(x))$, and
\[
  Q_{f,\rm SP}(\varphi^u,\alpha,c,n)\eqdef
  \exp\left(n\inf_{x\in\cR} \varphi^u(x)\right)
\]
otherwise. We define
\[
  P_{f,\rm SP} (\varphi^u,\alpha,c) \eqdef
  \limsup_{n\to\infty}\frac{1}{n}\log Q_{f,\rm SP}(\varphi^u,\alpha,c,n).
\]
It follows that if $\SFix_{(f,\alpha,c)}(f^n)\not=\emptyset$ for some $n\in\bN$ then
$P_{f,\rm SP} (\varphi^u,\alpha,c)$ is entirely determined by  the values of
$\varphi^u$ on the saddle points of $f$. Denote 
\[
P_{f,\rm SP} (\varphi^u)\eqdef
 \lim_{\alpha\to0}\lim_{c\to0}P_{f,\rm SP} (\varphi^u,\alpha,c).
\]

\subsection{The super-additive saddle point pressure}

We say that a sequence $\Psi=(\psi_n)_n$ of functions $\psi_n\colon M\to\bR$ is \emph{super-additive} (with respect to $f$) if for every $n$, $m\in\bN$ and $x\in M$ we have
\[
  \psi_n(x)+\psi_m(f^n(x))\le \psi_{n+m}(x) .
\]
For  a super-additive sequence $\Psi=(\psi_n)_n$ of continuous functions and for numbers  $n\in\bN$, $0<\alpha$ and $0<c\le 1$ we define
\[
Q_{f,\rm SP}(\Psi,\alpha,c,n) \eqdef
\sum_{x\in\SFix_{(f,\alpha,c)}(f^n)}\exp\left(\psi_n(x)\right) 
\]
if $\SFix_{(f,\alpha,c)}(f^n)\ne\emptyset$, and $Q_{f,\rm SP}(\Psi,\alpha,c,n) \eqdef
\exp(\inf_{x\in\Lambda}\psi_n(x))$ otherwise. We define 
\begin{equation}\label{def:nap}
P_{f,\rm SP}(\Psi,\alpha,c) \eqdef
\limsup_{n\to\infty}\frac{1}{n}\log Q_{f,\rm SP}(\Psi,\alpha,c,n).
\end{equation}
It follows that if $\SFix_{(f,\alpha,c)}(f^n)\ne\emptyset$ for some $n\in\bN$, then $P_{f,\rm SP}(\Psi,\alpha,c)$ is entirely determined by the values of the potentials $\psi_n$ on the periodic points of $f$ which are of saddle type.  Denote 
\[
P_{f,\rm SP} (\Psi)\eqdef
 \lim_{\alpha\to0}\lim_{c\to0}P_{f,\rm SP} (\Psi,\alpha,c).
\]

\section{Properties of the pressures}\label{sec:nonadd}

We collect together some immediate properties of the topological pressure and of the above defined pressures.
Let in the following $\Psi=(\psi_n)_n$ be a sequence of continuous functions which is super-additive with respect to a $C^{1+\varepsilon}$ diffeomorphism $f\colon M\to M$.

\begin{proposition}\label{huc}
	$P_{f|\Lambda}(\psi_1)\le P_{f|\Lambda}(\frac{1}{n}\psi_n)$ and 
	$P_{f^n|\Lambda}(S_n\psi_1)\le P_{f^n|\Lambda}(\psi_n)$, where here $S_n\psi_1(x) =\sum_{k=1}^{n-1}\psi_1(f^k(x))$.
\end{proposition}

\begin{proof} This follows immediately from super-additivity of the sequence $(\psi_n)_n$ and from  the variational principle for the topological pressure.
\end{proof}  

\begin{proposition}
	$P_{f,\rm SP}(\psi_1) \le 
	P_{f,\rm SP}\left(\frac{1}{n}\psi_n\right) \le
	P_{f,\rm SP}(\Psi)$.
\end{proposition}

\begin{proof} 
Fix $n\in\bN$. Let $m\in\bN$ and $0\le \ell<n$. From super-additivity of the function sequence we obtain
\begin{eqnarray*} 
&&\sum_{k=0}^{mn+\ell-1}\psi_n(f^k(x)) 
 =\sum_{i=0}^{n-1}\sum_{k=0}^{m-2}\psi_n(f^{kn+i}(x))
+\sum_{k=(m-1)n}^{mn+\ell-1}\psi_n(f^k(x))\\
&\ge& n\sum_{k=0}^{mn-1}\psi_1(f^k(x)) 
	-\sum_{k=0}^{n-1}\left(\psi_k(x)+\psi_{k}(f^{mn-k-1}(x))\right)
	+ (n+\ell-1)C_2(n)\\
&\ge& n\sum_{k=0}^{mn+\ell-1}\psi_1(f^k(x)) 	
	- n\sum_{k=mn}^{mn+\ell-1}\psi_1(f^k(x))
	- 2nC_1(n)
	+ (n+\ell-1)C_2(n)	\\
&\ge& n\sum_{k=0}^{mn+\ell-1}\psi_1(f^k(x)) 	
	- n(2+\ell)C_1(n)
	+ (n+\ell-1)C_2(n),
\end{eqnarray*}
where we set $\psi_0(y)=0$ and
\[
C_1(n)\eqdef\max_{k=1,\ldots,2n-1}\max_{x\in\Lambda}\psi_k(x), \quad
C_2(n)\eqdef \min_{k=1,\ldots,2n-1}\min_{x\in \Lambda}\psi_k(x).
\]  
From here we can conclude the first inequality is true.  
Analogously, by super-additivity we have
\begin{eqnarray*}
n\,\psi_{mn+\ell}(x)
& \ge&  \sum_{k=0}^{(m-1)n-1}\psi_n(f^k(x)) 
	+ \sum_{k=0}^{n-1}\left( \psi_k(x) +\psi_{n+\ell-k}(f^{(m-1)n+k}(x) \right)\\
& \ge& \sum_{k=0}^{mn+\ell-1}\psi_n(f^k(x))+2 n\,C_2(n)	
	- \sum_{k=0}^{n+\ell-1}\psi_n(f^{(m-1)n+k}(x))\\
&\ge& \sum_{k=0}^{mn+\ell-1}\psi_n(f^k(x))
	-(n+\ell)\, C_1(n)	+2nC_2(n).
\end{eqnarray*}
From here we can conclude the second inequality. 
\end{proof}

\begin{proposition}\label{lem:gh}
  Let $\alpha>0$, $0< c< 1$, and  $n\ge \bN$ such that $\SFix_{(f,\alpha,c)}(f^n)\ne\emptyset$. Then
  \[
  P_{f,\rm SP}(\Psi,\alpha,c) \ge \frac{1}{n}P_{f^n,\rm SP}(\psi_n,n\alpha,c\cdot d(n))
  \]
  with $d(n)$ defined in~(\ref{kubus}).
\end{proposition}
  
\begin{proof}
Let $n\in\bN$ and $0<\alpha<1$ and $0<c<1$ such that $\SFix_{(f,\alpha,c)}(f^n)\ne\emptyset$. Given $m\in\bN$, by super-additivity and by the inclusion~(\ref{kubu}) we have
\begin{eqnarray*}
&& Q_{f,\rm SP}(\Psi,\alpha,c,mn) 
= \sum_{x\in\SFix_{(f,\alpha,c)}(f^{mn})}\exp\left(\psi_{mn}(x)\right)\\
& \ge& \sum_{x\in\SFix_{(f^n,n\alpha,c\cdot d(n))}(f^{mn})}
	\exp\left(\psi_{mn}(x)\right)\\
&\ge& \sum_{x\in\SFix_{(f^n,n\alpha,c\cdot d(n))}((f^n)^m)}
	\exp\left(\psi_n(x)+
\psi_n(f^n(x))+\cdots +\psi_n(f^{(m-1)n}(x))\right).
\end{eqnarray*}
Thus, letting $m\to\infty$, we obtain
\[
P_{f,\rm SP}(\Psi,\alpha,c) 
\ge \frac{1}{n} P_{f^n,\rm SP}(\psi_n,n\alpha,c\cdot d(n)).
\]
This proves the proposition.
\end{proof}

 
We now introduce the sequence $\Phi=(\varphi_n)_n$ of potentials which measure the volume growth under $df^n$. Consider the linear map $(df_x)^\wedge\colon (T_xM)^\wedge\to (T_{f(x)}M)^\wedge$ between the full exterior algebras of the tangent spaces induced by $df_x$. Let  $\varphi_n\colon\Lambda\to\bR$ be given by  
\begin{equation}\label{hop}
\varphi_n(x) \eqdef  -\log|| (df^n_x)^\wedge ||,
\end{equation}
where the norm is the one that is induced by the Riemannian metric.
Denote by $\vol_\ell(v_1,\ldots, v_\ell;x)$ the $\ell$-dimensional volume of a
parallelepiped which is spanned by the vectors $v_1$, $\ldots$, $v_\ell\in
T_xM$. Notice that we have 
\[
|| (df^n_x)^\wedge || 
= \max_{1\le \ell\le \dim M}\sup_{v_i\in T_xM}
\frac{\vol_\ell(df^n_x(v_1),\ldots,df^n_x(v_\ell);f^n(x))}
     {\vol_l(v_1\ldots,v_\ell;x)} .
\]

\begin{lemma}\label{lem:2}
  The sequence $\Phi=(\varphi_n)_n$ given by~(\ref{hop}) is a super-additive
  sequence of H\"older continuous functions.
\end{lemma}

\begin{proof}
  We can express $|| (df^n_x)^\wedge||$ in terms of the
  singular values of $df^n_x$. Recall that the \emph{singular values}
  $\sigma_1(L)\ge\cdots\ge\sigma_{\dim M}(L)\ge 0$ of a linear operator
  $L\colon T_xM\to T_{f(x)}M$ are the eigenvalues of 
  $(L^\ast L)^{1/2}$, where $L^\ast$ denotes the adjoint of $L$. We have 
  \begin{equation}\label{huj}
    || (df^n_x)^\wedge|| 
    = \max_{1\le \ell\le\dim M}\sigma_1(df^n_x)\cdots\sigma_\ell(df^n_x).
  \end{equation}
  For every $n\in\bN$ and every $1\le \ell\le \dim M$ the map $x\mapsto \sigma_\ell(df^n_x)$ is
  H\"older continuous on $M$. Given $x\in\Lambda$, for some number $1\le \ell=\ell(x,n+m)\le\dim M$ we have $|| (df^{n+m}_x)^\wedge|| = \sigma_1(df^{m+n}_x)\cdots\sigma_\ell(df^{m+n}_x)$, which implies 
  \begin{eqnarray*}
    || (df^{n+m}_x)^\wedge|| 
   & =& \sigma_1(df^{n+m}_x)\cdots\sigma_\ell(df^{n+m}_x)\\
    	    &\le& \sigma_1(df^n_x)\cdots\sigma_\ell(df^n_x)\cdot
    \sigma_1(df^m_{f^n(x)})\cdots\sigma_\ell(df^m_{f^n(x)})\\
    &\le& || (df^n_x)^\wedge|| \,|| (df^m_{f^n(x)})^\wedge|| ,
  \end{eqnarray*}
  where the last inequality follows from the relation~(\ref{huj}).
  This implies the super-additivity of the sequence $\Phi=(\varphi_n)_n$.
\end{proof}

The exponential volume growth rate is naturally related to the Lyapunov exponents of $f$.  For every Lyapunov regular point $x$ with a positive Lyapunov exponent, by $df$-invariance of the unstable subbundle $E^u$ we conclude that  
\begin{equation}\label{sontag}
- \log| \det df^n_x|_{E^u_x}| = S_n\varphi^u(x),
 \end{equation}
moreover, for every Lyapunov regular point  we have
\[
  \lim_{n\to\infty}\frac{1}{n}\log|\det df^n_x|_{E^u_x}| =
  \sum_i \lambda_i(x)^+,
\]
where we denote $a^+\eqdef\max\{0,a\}$ (see~\cite[Corollary~11.4]{Man:}).

We denote by $\cM^\infty_{\rm E}$ the set of all ergodic Borel probability measures that are invariant with respect to $f^k$ for some $k\ge 1$. Let $\nu\in\cM^\infty_{\rm E}$ be an ergodic  $f^k$-invariant Borel probability measure. 
Super-additivity implies the existence of the following limit and the equality 
\begin{equation}\label{ko}
  \lim_{n\to\infty}\frac{1}{kn}\int_\Lambda \varphi_{kn}\, d\nu
= \sup_{n\ge 1} \int_\Lambda\frac{1}{kn} \varphi_{kn} d\nu 
\end{equation}
(compare, for example,~\cite[Theorem 10.1]{Wal:81}).  
%
%
Further, since for $1\le \ell < k$  
\[
\int_\Lambda\varphi_{kn+\ell}\,d\nu \ge 
\int_\Lambda (\varphi_{kn}(x) + \varphi_\ell(f^{kn}(x)) )\, d\nu (x)
= \int_\Lambda \varphi_{kn}\, d\nu + \int_\Lambda\varphi_\ell \,d\nu 
\]
and 
\[
\int_\Lambda\varphi_{kn+\ell}\,d\nu
\le \int_\Lambda\varphi_{(n+1)k}\,d\nu-\int_\Lambda\varphi_{k-\ell}\,d\nu,
\]
we can conclude that the following limit exists 
\begin{equation}\label{eq:wqw} 
\lim_{n\to\infty}\frac{1}{n}\int_\Lambda\varphi_n d\nu
\end{equation}
and is equal to~(\ref{ko}).

If $\nu\in\cM^\infty_{\rm E}$ is an ergodic  $f^k$-invariant Borel probability measure
with a positive Lyapunov exponent, then it follows from the multiplicative ergodic theorem that the function $x\mapsto \log|\det df^k_x|_{E^u_x}|$ is integrable and  that $\nu$-almost every point $x$ is Lyapunov regular and satisfies
\[
 \lim_{n\to\infty}\frac{1}{n} \log|| (df^{kn}_x)^\wedge || 
= \lim_{n\to\infty}\frac{1}{n}\log|\det df^{kn}_x|_{E^u_x}|
=  k\sum_i\lambda_i(x)^+ 
\]
(see for example~\cite{Man:}). Moreover, the functions $x\mapsto \lambda_i(x)$ are measurable and $f^k$-invariant, and so constant $\lambda_i(x)=\lambda_i(\nu)$ for  $\nu$-almost every $x$ and
\begin{equation}\label{kooo}
  \lim_{n\to\infty}\frac{1}{kn}\int_\Lambda \varphi_{kn}\, d\nu
=\sup_{n\ge 1} \int_\Lambda\frac{1}{kn} \varphi_{kn}\, d\nu 
= \int_\Lambda\varphi^u\, d\nu = -\sum_i\lambda_i(\nu)^+   .
\end{equation}
 %
%
Since for every $n\ge 1$
\[
\frac{1}{k}\int_\Lambda\varphi_k\,d\nu = 
\frac{1}{kn}\int_\Lambda n\varphi_k\,d\nu 
\le
\frac{1}{kn}\int_\Lambda\varphi_{kn}\,d\nu
\]
(remember that $\nu$ is $f^k$-invariant), with~(\ref{kooo}) we obtain 
\begin{equation}\label{martina}
\frac{1}{k}\int_\Lambda\varphi_k\,d\nu \le 
-\sum_i\lambda_i(\nu)^+.
\end{equation}
Furthermore,  even though that $\nu$ may not be invariant with respect to the map $f$, it makes sense to define the \emph{entropy} of $\nu$ with respect to $f$ by
\[
h_\nu(f)\eqdef \frac{1}{k}h_\nu(f^k)
\]
(see~\cite{Fal:88} for details).


We are now prepared to prove Theorem~\ref{theorem:heu}.

\begin{proof}[Proof of Theorem~\ref{theorem:heu}] Observe that for any $f$-invariant ergodic measure $\mu$ (which, by assumption, possesses a positive Lyapunov exponent) as a particular case of~(\ref{kooo})  we have
\begin{equation}\label{eli}
\lim_{n\to\infty}\int_\Lambda\frac{1}{n}\varphi_{n}\,d\mu
= \sup_{n\ge 1} \int_\Lambda\frac{1}{n} \varphi_{n}\, d\mu
= \int_\Lambda\varphi^ud\mu.
\end{equation}
Thus, using the variational principle for pressure, we obtain
\[
h_\mu(f)+ \int_\Lambda\varphi^ud\mu =
h_\mu(f)+\sup_{n\ge 1}\int_\Lambda\frac{1}{n}\varphi_{n}\,d\mu \le 
\sup_{n\ge 1}P_{f|\Lambda}\left(\frac{1}{n}\varphi_n\right).
\]
On the other hand for every $n\ge 1$
\[
P_{f|\Lambda}\left(\frac{1}{n}\varphi_n\right) = 
\sup_{\mu\in\cM_{\rm E}} 
\left(h_\mu(f)+\int_\Lambda\frac{1}{n}\varphi_{n}\,d\mu \right)
\le \sup_{\mu\in\cM_{\rm E}} \left(h_\mu(f)+\int_\Lambda\varphi^ud\mu \right),
\]
where we used~(\ref{eli}).
 From here the statement follows. 
\end{proof}

\begin{proposition}\label{gigi}
$P_{f,\rm SP}(\Phi)\le 	P_{f,\rm SP}(\varphi^u)$
\end{proposition}

\begin{proof}
	Notice that we have $\varphi_n(x) \le  S_n\varphi^u(x)$ for every saddle point $x$.
\end{proof}

\begin{proposition}\label{lem:main}
	If $f|\Lambda$ is non-uniformly hyperbolic, then 
	$P_{f,\rm SP}(\frac{1}{n}\varphi_n)=P_{f|\Lambda}(\frac{1}{n}\varphi_n)$ and $P_{f^n,\rm SP}(\varphi_n)=P_{f^n|\Lambda}(\varphi_n)$.
\end{proposition}

\begin{proof}
For every $n\ge 1$, $f^n$ is a $C^{1+\varepsilon}$ diffeomorphism which is non-uniformly hyperbolic and with respect to which $\Lambda$ is locally maximal.	
Given the H\"older continuous potentials $\frac{1}{n}\varphi_n$ and $\varphi_n$,  the statement is proved similar to~\cite[Theorem 1]{GelWol:} applied to the diffeomorphisms $f$ and $f^n$, respectively.
\end{proof}  

\begin{proposition}\label{his}
	Under the hypothesis of Theorem~\ref{theorem:2} we have 
	\begin{equation}\label{martu}
	\sup_{\nu\in\cM^\infty_{\rm E}}\left( h_\nu(f)-\sum_i\lambda_i(\nu)^+
	\right) 
	\le P_{f,\rm SP}(\Phi).
	\end{equation}
\end{proposition}

\begin{proof}
Let us assume that there exists $\nu\in\cM_{\rm E}^\infty$ which is invariant with respect to $f^k$ for some $k\ge 1$ and which satisfies
\[
0< h_\nu(f) -\sum_i\lambda_i(\nu)^+ - 
P_{f,\rm SP}(\Phi)\eqdef \delta.
\]	
By~(\ref{kooo}) there exists $N\ge 1$ such that for every $\ell\ge N$
\[
h_\nu(f) + \frac{1}{\ell}\int_\Lambda \varphi_\ell d\nu \ge P_{f,\rm SP}(\Phi) +\frac{3}{4}\delta.
\]
Note that any $f^k$-invariant ergodic measure is also invariant with respect to $f^{k\ell}$, $\ell\ge 1$.
Hence, applying the variational principle for the topological pressure with respect to the map $f^{k\ell}$, we obtain for every $\ell\ge N$
\[
\frac{1}{k\ell}P_{f^{k\ell}|\Lambda}(\varphi_{k\ell}) 
\ge h_\nu(f)+\frac{1}{k\ell}\int_\Lambda\varphi_{k\ell}d\nu
\ge P_{f,\rm SP}(\Phi)+\frac{3}{4}\delta.
\]
With Proposition~\ref{lem:main} we obtain
\[
\frac{1}{k\ell}P_{f^{k\ell}, \rm SP}(\varphi_{k\ell}) 
\ge P_{f,\rm SP}(\Phi)+\frac{3}{4}\delta.
\]
Choose numbers $\alpha_0>0$, and $0<c_0<1$ such that
\[
P_{f,\rm SP}(\Phi) +\frac{1}{4}\delta
\ge  P_{f,\rm SP}(\Phi,\alpha,c).
\]
for every $0<\alpha<\alpha_0$ and $0<c<c_0$. 

	Recall that by the Katok closing lemma (see~\cite{KatHas:95}), given a Borel probability measure $\nu$ which is invariant and hyperbolic with respect to some $C^{1+\varepsilon}$ diffeomorphism $g\colon M\to M$, we can always find sufficiently many orbits that are periodic with respect to $g$ (with some period) and of saddle type. In particular, given a $g= f^k$-invariant hyperbolic measure, for sufficiently large $\ell\ge1$ the set $\SPer(f^{k\ell})$ is nonempty.  
Fixing now $\ell\ge N$ such that $\SPer(f^{k\ell})\ne\emptyset$, we find numbers $\alpha=\alpha(\ell)<\alpha_0$ and $c=c(\ell)<c_0$ small enough such that $\SPer_{(f,\alpha,c)}(f^{k\ell})\ne\emptyset$ and that
\[
\frac{1}{k\ell}P_{f^{k\ell},\rm SP}(\varphi_{k\ell},kl\alpha,c\cdot d(k\ell))
\ge \frac{1}{k\ell}P_{f^{k\ell},\rm SP}(\varphi_{k\ell}) -\frac{\delta}{4}
\]
and hence
\[
\frac{1}{k\ell}P_{f^{k\ell},\rm SP}(\varphi_{k\ell},k\ell\alpha,c\cdot d(k\ell))
> P_{f,\rm SP}(\Phi,\alpha,c), 
\]
in contradiction to Proposition~\ref{lem:gh}. Thus, we have shown~(\ref{martu}).
\end{proof}

\begin{proof}[Proof of Theorem~\ref{theorem:2}]
With Proposition~\ref{gigi} and with Theorem~\ref{Main} we obtain
\[
P_{f,\rm SP}(\Phi)  
\le P_{f,\rm SP}(\varphi^u) 
= \sup_{\mu\in\cM_{\rm E}}
	\left(h_\mu(f)+\int_\Lambda\varphi^u d\mu\right).
\]
Clearly $\cM_{\rm E}\subset \cM_{\rm E}^\infty$. Thus, with Proposition~\ref{his} we can conclude that
\begin{eqnarray*}
\sup_{\nu\in\cM^\infty_{\rm E}}\left(h_\nu(f)-\sum_i\lambda_i(\nu)^+
\right) 
&=& \sup_{\mu\in\cM_{\rm E}}\left(h_\mu(f)+\int_\Lambda \varphi^ud\mu\right) \\
&=& P_{f,\rm SP}(\Phi)  .
\end{eqnarray*}
Moreover, with  Theorem~\ref{theorem:heu} and with Proposition~\ref{lem:main} we obtain
\[
\sup_{\mu\in\cM_{\rm E}}\left(h_\mu(f)+\int_\Lambda \varphi^u d\mu\right)
=\sup_{n\ge 1}P_{f|\Lambda}\left(\frac{1}{n}\varphi_n\right)
= \sup_{n\ge 1}P_{f,\rm SP}\left(\frac{1}{n}\varphi_n\right) .
\]
This proves the theorem.
\end{proof}

We finally discuss slightly weaker hypotheses~(\ref{bedung}) and (\ref{bedung2}) under which (\ref{eq:www}) can be established: 
Let $f\colon M\to M$ be a $C^{1+\varepsilon}$ surface diffeomorphism and let $\Lambda\subset M$ be a compact locally maximal $f$-invariant set. Assume that
\begin{equation}\label{bedung}
	 \sup_{\mu\in\cM_{\rm E}} \frac{1}{n} \int_\Lambda \varphi_n \,d\mu
	 <  P_{f|\Lambda}\left(\frac{1}{n}\varphi_n\right) 
	 \,\,{\rm for}\,\,{\rm every }\,\,n\ge 1
\end{equation}
or
\begin{equation}\label{bedung2}
- \inf_{\mu\in\cM_{\rm E}}\sum_i\lambda_i(\mu)^+ 
< P_{f|\Lambda}\left(\frac{1}{n}\varphi_n\right)
 \,\,{\rm for}\,\,{\rm every }\,\,n\ge 1
\end{equation}
in which case, using~(\ref{ko}) 
 we can also conclude that for $n\ge 1$
\[
 \sup_{\mu\in\cM_{\rm E}} \frac{1}{n} \int_\Lambda \varphi_n \,d\mu
 \le - \inf_{\mu\in\cM_{\rm E}}\sum_i\lambda_i(\mu)^+
< 
P_{f|\Lambda}\left(\frac{1}{n}\varphi_n\right).
\]
This allows us to apply~\cite[Theorem 1]{GelWol:} to the diffeomorphism $f$ and the potential $\frac{1}{n}\varphi_n$. We obtain $P_{f,\rm SP}(\frac{1}{n}\varphi_n)=P_{f|\Lambda}(\frac{1}{n}\varphi_n)$ for every $n\ge 1$ and hence the  second equality in~(\ref{eq:www}). 
We only mention that similar hypothesis can be stated on the measures in $\cM^\infty_{\rm E}$ in order to establish the first equality in~(\ref{eq:www}), but will refrain from formulating them since it seems to be hard to control the total set of measures~$\cM^\infty_{\rm E}$ in general. 

\section{Proof of Theorem~\ref{Main}}\label{sec:2}

Consider $\mu\in \cM_{\rm E}$ hyperbolic with $h_\mu(f)>0$,  and let $0<\alpha<\chi(\mu)$. Extending some results by Katok (see e.g.~\cite[Chapter S.5]{KatHas:95}),  or using Mendoza~\cite{Men:88} in the case of a surface diffeomorphism and S\'anchez-Salas~\cite{San:02} in the case of a diffeomorphisms on a higher-dimensional manifold, we derive the  existence of a sequence $(\mu_n)_n$ of measures $\mu_n\in \cM_{\rm E}$ supported on hyperbolic horseshoes
$K_n(\mu)\subset M$ (see~\cite{KatHas:95} for the definition) such that $h_{\mu_n}(f)\to h_\mu(f)$, and $\chi(\mu_n)\to\chi(\mu)$. In particular, for each
$n\in\bN$ there exist $m,s\in\bN$ such that $f^m| K_n(\mu)$ is conjugate to the full shift in $s$ symbols. Since $\Lambda$ is a compact locally maximal $f$-invariant set, we can conclude that
$K_n(\mu)\subset \Lambda$ for all $n\in \bN$. 
The  following is a consequence of  the construction given in these papers:   One can construct the horseshoe $K_n(\mu)$ in such a way that $|\int\varphi^ud\mu + \log|\det df^m_x|_{E^u_x}|^{1/m}|$ is small for points  $x$ in the rectangles covering $K_n(\mu)$; and  the variation of $\log| \det df^m_x|_{E^u_x}|$ on the horseshoe $K_n(\mu)$ is bounded by a constant which only depends on the diameter of the rectangle cover of the Pesin set, and which tends to zero as $n\to\infty$. 
Then we can conclude that for every $0<\varepsilon<\chi(\mu)-\alpha$ there is a number $n=n(\varepsilon)\ge 1$ and a measure $\mu\in\cM_{\rm E}$ supported on $K_n(\mu)$ such that
\begin{equation}\label{nag}
  h_\mu(f)-\varepsilon < h_{\mu_n}(f) ,\quad
  \left| \int_\Lambda\varphi^ud\mu_n - \int_\Lambda\varphi^ud\mu\right| < \varepsilon.
\end{equation}
Moreover, there exists a number $0<c_0(n)\le1$ such that for
every periodic point $x\in K_n(\mu)$ and every $k\in\bN$ we have
\begin{eqnarray*}
 {\rm for}\,\,{\rm every }\,\,v\in E^u_x \,\,\,
  || df^k_x(v)|| &\ge & c_0(n)e^{k(\chi(\mu)-\varepsilon)}|| v||\\
  {\rm for}\,\,{\rm every }\,\,w\in E^s_x\,\,\,
  || df^{-k}_x(w)|| &\le& c_0(n)e^{k(-\chi(\mu)+\varepsilon)}||
  w||.
\end{eqnarray*}
This implies that for every $k\in\bN$
\begin{equation}\label{kir}
  \Fix(f^k)\cap K_n(\mu) \subset \SFix_{(f,\alpha,c_0)}(f^k).
\end{equation}
It follows from the above given approximation properties of the measures $\mu_n$ that 
\begin{equation}\label{marta}
h_\mu(f)+\int\varphi^ud\mu \le 
\sup_{n\ge 1}\left(h_{\mu_n}(f)+\int\varphi^ud\mu_n\right)
\le \sup_{n\ge 1} P_{f|K_n(\mu)}(\varphi^u).
\end{equation}
From~\cite[Proposition 20.3.3]{KatHas:95} we derive that for every $n\ge 1$
\[
P_{f|K_n(\mu)}(\varphi^u) = 
\lim_{k\to\infty}\frac{1}{k}\log\sum_{x\in\Fix(f^k)\cap K_n(\mu)}| \det df^k_x|_{E^u_x}|^{-1}.
\]
With~(\ref{kir}), for $0<\alpha<\chi(\mu)$ we conclude that
\begin{eqnarray*}
h_\mu(f)+\int\varphi^ud\mu &\le &
\limsup_{k\to\infty}\frac{1}{k}\log\sum_{x\in\SFix_{(f,\alpha,c_0(n))}(f^k)}| \det df^k_x|_{E^u_x}|^{-1}\\
&\le& \lim_{c\to 0}P_{f,\rm SP}(\varphi^u,\alpha,c) 
\le P_{f,\rm SP}(\varphi^u).
\end{eqnarray*}

Consider now a hyperbolic measure $\mu\in\cM_{\rm E}$ with $h_\mu(f)=0$. By~\cite[Theorem S.5.5]{KatHas:95}, for every $x\in{\rm supp}\,\mu$ and every $\delta>0$ there exists a hyperbolic periodic point $z$, say of period $m$, in any small neighborhood $V$ of $x$. In particular, one derives that  $|\varphi^u(z)+\log| \det df^m_x|_{E^u_x}|^{1/m}|$ is small. By local maximality of $\Lambda$ we conclude $z\in\Lambda$. From here we obtain for $0<\alpha<\chi(\mu)$
\[
\int\varphi^u d\mu \le \lim_{c\to 0}P_{f,\rm SP}(\varphi^u,\alpha,c)
\le P_{f,\rm SP}(\varphi^u).
\]

With the above we can conclude that 
\[
\sup_{\mu \in\cM_{\rm E}\,\,{\rm hyperbolic}}\left(h_\mu(f)+\int\varphi^ud\mu\right)
\le P_{f,\rm SP}(\varphi^u).
\]
It follows from~\cite[Theorem 4]{GelWol:} that
\[
\lim_{c\to0}P_{f,\rm SP}(\varphi^u,\alpha,c) \le
\sup\left(h_\nu(f)+\int\varphi^ud\nu\right),
\]
where the supremum is taken over all hyperbolic measures $\nu\in\cM_{\rm E}$ with $\alpha\le\chi(\nu)$. 

The remaining equality follows with  Theorem~\ref{theorem:heu}. This proves Theorem~\ref{Main}.

\end{document}